\documentclass[fleqn,preprint,3p,a4paper]{elsarticle}

\usepackage{amsmath,amssymb,amsthm,mathtools}
\usepackage{microtype}

\newtheorem{theorem}{Theorem}[section]

\newtheorem{lemma}[theorem]{Lemma}
\newtheorem{corollary}[theorem]{Corollary}
\theoremstyle{definition}

\newtheorem{example}[theorem]{Example}
\newtheorem{remark}[theorem]{Remark}

\newcommand{\PS}{P_{\!S}}
\newcommand{\Q}{\mathcal Q}
\newcommand{\KIrr}{\operatorname{K\!\text{-}Irr}}
\newcommand{\OO}{\mathcal O}
\newcommand{\cl}{\operatorname{cl}}
\newcommand{\up}[1]{\mathop{\uparrow}\!_{#1}}

\journal{Topology and its Applications}

\begin{document}
\begin{frontmatter}

\title{Compact localization and co-sobriety of Smyth power spaces\tnoteref{t1}}
\tnotetext[t1]{This research was supported by the National Natural Science Foundation of China  (Nos. 12471070, 12071199, 12561083).}

\author[1]{Wei Ji}
\ead{javeey@163.com}
\author[2]{Xiaoquan Xu\corref{cor1}}
\cortext[cor1]{Corresponding author}
\ead{xiqxu2002@163.com}

\address[1]{School of Mathematics and Statistics, Guilin University of Technology, Guilin 541004, China}
\address[2]{School of Computer Information Engineering, Nanchang Institute of Technology, Nanchang 330044, China}

\begin{abstract}
We establish a compact-localization theorem for finite iterates of the Smyth power space. Successive union maps assign to every finite-level Smyth object a canonical least compact saturated carrier in the original space, and co-sobriety at each positive finite level is equivalent to co-sobriety of the corresponding iterates over all compact saturated carriers. Consequently, if every compact saturated subspace of a $T_0$-space $X$ is stably compact---in particular, if every such subspace is Hausdorff---then $\PS^n(X)$ is co-sober for every finite $n\geq 1$. This applies to every $k_1$-Hausdorff space and hence to every Hausdorff or metrizable space. For the first iterate, we also give a direct full-Vietoris argument, independent of stable compactness, which shows that each $k$-irreducible compact saturated family is generated by its union.
\end{abstract}

\begin{keyword}
Co-sober space \sep compact localization \sep compact saturated set \sep Smyth power space \sep $k_1$-Hausdorff space \sep stably compact space
\MSC[2020] 54B20 \sep 54D10 \sep 54D30 \sep 06B35
\end{keyword}

\end{frontmatter}

\section{Introduction}

The Smyth power construction is one of the standard powerdomain constructions in domain theory and non-Hausdorff topology; see, for example, \cite{AbramskyJung,GierzEtAl,GoubaultBook,Schalk,XuWenXi}. For a $T_0$-space $X$, its Smyth power space $\PS(X)$ consists of the nonempty compact saturated subsets of $X$, equipped with the upper Vietoris topology. Several fundamental topological properties interact well with this construction: sobriety is preserved and reflected \cite{HeckmannKeimel,Schalk}, while stable compactness is preserved \cite{Goubault2010}.

Co-sobriety behaves more delicately. It was introduced by Escard\'o, Lawson and Simpson \cite{EscardoLawsonSimpson} in connection with dual forms of the Hofmann--Mislove theorem and compactly generated spaces. A nonempty compact saturated subset is called $k$-irreducible if it is not the union of two proper compact saturated subsets, and a $T_0$-space is co-sober if every $k$-irreducible compact saturated subset is the saturation of a unique point. Despite the formal duality with sobriety, the two notions differ substantially: a sober space need not be co-sober \cite{ShenWuXiZhao,WenXu}. More recently, Xu, Xu, Ji and Qiu constructed a co-sober $T_1$-space whose Smyth power space is not co-sober \cite{XuXuJiQiu}, so co-sobriety is not preserved by the Smyth power construction in general.

Positive preservation results therefore require additional structure. Xiao and Jia recently established a preservation-and-reflection theorem for coherence together with weak Hausdorffness \cite{XiaoJia}. That result does not by itself settle the Hausdorff case, since the available co-sobriety criterion also involves local compactness, whereas Hausdorff spaces need not be locally compact. The Hausdorff and metrizable cases were also explicitly raised in a recent invited presentation by Xu.\footnote{X.~Xu, \emph{Some open problems on co-sober spaces}, invited academic presentation at the 11th International Symposium on Domain Theory and its Applications, Chengdu, China, 4--8 August 2026.}

The main tool of this paper is \emph{compact localization}. Repeated union flattens an object of a finite Smyth iterate to a compact saturated subset of the original space. We prove that this subset is its canonical least carrier: for every finite $n\geq1$, an object $A\in\PS^n(X)$ belongs to $\PS^n(K)$ precisely when its carrier is contained in $K$. This yields the localization equivalence
\[
 \PS^n(X)\text{ is co-sober}
 ~\Longleftrightarrow~
 \PS^n(K)\text{ is co-sober for every }K\in\Q(X).
\]
Thus the global problem is reduced to compact saturated subspaces of the original space.

This localization immediately gives a general sufficient condition: if every $K\in\Q(X)$ is stably compact, then every positive finite Smyth iterate of $X$ is co-sober. In particular, it is enough that every compact saturated subspace of $X$ be Hausdorff. Hence every $k_1$-Hausdorff space has co-sober finite Smyth iterates, and the same is true, in particular, for every Hausdorff or metrizable space.

For the first iterate we give a second, direct argument on compact Hausdorff carriers that does not use preservation of stable compactness. An upper-Vietoris compact saturated family is closed in the full Vietoris hyperspace. A finite decomposition argument then shows that every $k$-irreducible family contains its union. Consequently, whenever every compact saturated subspace of $X$ is Hausdorff and $\mathcal A\in\KIrr(\PS(X))$,
\[
  \mathcal A=\up{\PS(X)}\!\left(\bigcup\mathcal A\right),
\]
so the unique generating point is identified explicitly.

Section~2 fixes notation and recalls the required background. Section~3 develops iterated carriers and proves the compact-localization theorem. Section~4 gives the direct Vietoris argument for compact Hausdorff carriers. Section~5 derives the stable-compact-carrier criterion and its $k_1$-Hausdorff, Hausdorff and metrizable consequences. The final section records the scope and limitations of the method.

\section{Preliminaries}

Let $X$ be a $T_0$-space. Its specialization order is
\[
   x\leq_X y ~\Longleftrightarrow~ x\in\cl_X\{y\}.
\]
For $A\subseteq X$, write
\[
   \up{X}A=\{x\in X:a\leq_Xx\text{ for some }a\in A\}.
\]
A subset is saturated if it is an upper set in the specialization order, equivalently, an intersection of open sets. We write $\Q(X)$ for the set of all nonempty compact saturated subsets of $X$. The space $X$ is called \emph{coherent} if the intersection of any two compact saturated subsets is compact.

A member $G\in\Q(X)$ is \emph{$k$-irreducible} if
\[
   G=G_1\cup G_2,~ G_1,G_2\in\Q(X),
\]
implies $G=G_1$ or $G=G_2$. The set of all such members is denoted by $\KIrr(X)$. The space $X$ is \emph{co-sober} if every $G\in\KIrr(X)$ has the form $G=\up{X}x$ for a unique $x\in X$.

\begin{lemma}[{\cite[Proposition~4.7 and the remark below Definition~3]{ZhaoHeWang}}]\label{lem:lc-coherent}
Every locally compact coherent $T_0$-space is co-sober.
\end{lemma}

For a $T_0$-space $X$ and $U\in\OO(X)$, let
\[
  \Box_X U=\{K\in\Q(X):K\subseteq U\}.
\]
These sets form a base for the upper Vietoris topology on $\Q(X)$. The resulting space is the \emph{Smyth power space} $\PS(X)$. Its specialization order is reverse inclusion:
\begin{equation}\label{eq:smyth-order}
   K\leq_{\PS(X)}L ~\Longleftrightarrow~ L\subseteq K.
\end{equation}

A nonempty subset $A$ of a $T_0$-space $X$ is \emph{irreducible} if, for closed subsets $F_1,F_2$,
\[
   A\subseteq F_1\cup F_2
   ~\Longrightarrow~
   A\subseteq F_1\text{ or }A\subseteq F_2.
\]
The space $X$ is \emph{sober} if every irreducible closed subset $F$ of $X$ is of the form $F=\cl_X\{x\}$ for a unique $x\in X$.

\begin{theorem}[{\cite[Theorem~3.13]{HeckmannKeimel}; see also \cite[Lemma~7.20]{Schalk}}]\label{thm:sob-smyth}
For a $T_0$-space $X$, the following conditions are equivalent:
\begin{enumerate}[\rm (1)]
\item $X$ is sober.
\item $\PS(X)$ is sober.
\end{enumerate}
\end{theorem}

We use repeatedly the following elementary facts.

\begin{lemma}\label{lem:saturated-subspace}
Let $Y$ be a saturated subspace of a $T_0$-space $X$. Then the identity map induces a homeomorphism
\[
   \PS(Y)\cong\{K\in\Q(X):K\subseteq Y\},
\]
where the set on the right has the subspace topology inherited from $\PS(X)$.
\end{lemma}

\begin{proof}
If $K\in\Q(Y)$, then $K$ is compact in $X$. If $x\in K$ and $x\leq_X z$, saturation of $Y$ gives $z\in Y$, and saturation of $K$ in $Y$ gives $z\in K$. Thus $K\in\Q(X)$ and $K\subseteq Y$. Conversely, if $K\in\Q(X)$ and $K\subseteq Y$, then compactness passes to the subspace $Y$, and saturation in $X$ implies saturation in $Y$. Hence $K\in\Q(Y)$.

The two spaces therefore have the same underlying set. If $V=U\cap Y$ is open in $Y$, where $U\in\OO(X)$, then
\[
   \Box_YV=\Box_XU\cap\PS(Y).
\]
Hence the upper Vietoris topologies agree.
\end{proof}

\begin{lemma}\label{lem:kirr-subspace}
Let $Y$ be a saturated subspace of $X$, and let $G\in\KIrr(X)$ satisfy $G\subseteq Y$. Then $G\in\KIrr(Y)$.
\end{lemma}

\begin{proof}
By Lemma~\ref{lem:saturated-subspace}, $G\in\Q(Y)$. If $G=G_1\cup G_2$ with $G_i\in\Q(Y)$, then each $G_i$ is compact saturated in $X$, again by Lemma~\ref{lem:saturated-subspace}. The $k$-irreducibility of $G$ in $X$ yields $G=G_1$ or $G=G_2$.
\end{proof}

The following inheritance property is due to Wen and Xu \cite{WenXu}; the proof is included for later use.

\begin{lemma}\label{lem:cosober-saturated}
Every saturated subspace of a co-sober $T_0$-space is co-sober.
\end{lemma}

\begin{proof}
Let $Y$ be saturated in a co-sober space $X$, and let $G\in\KIrr(Y)$. By Lemma~\ref{lem:saturated-subspace}, $G\in\Q(X)$. If $G=G_1\cup G_2$ with $G_i\in\Q(X)$, then $G_i\subseteq G\subseteq Y$, and another application of Lemma~\ref{lem:saturated-subspace} gives $G_i\in\Q(Y)$. Hence $G$ is $k$-irreducible in $X$. There is therefore an $x\in X$ such that $G=\up{X}x$. Since $x\in G\subseteq Y$ and $Y$ is saturated, $\up{X}x\subseteq Y$, and hence $G=\up{Y}x$.
\end{proof}

For $K\in\Q(X)$, Lemma~\ref{lem:saturated-subspace} and \eqref{eq:smyth-order} give
\begin{equation}\label{eq:PSK-principal}
   \PS(K)\cong\{L\in\Q(X):L\subseteq K\}=\up{\PS(X)}K.
\end{equation}

The union operation on Smyth power spaces is standard; see, for example, \cite[Proposition~7.21]{Schalk}. We include the short proof because the same argument underlies the localization theorem.

\begin{lemma}\label{lem:union-map}
For every $T_0$-space $X$, the map
\[
   \mu_X:\PS(\PS(X))\longrightarrow\PS(X),
   ~
   \mu_X(\mathcal A)=\bigcup_{K\in\mathcal A}K,
\]
is well defined and continuous.
\end{lemma}

\begin{proof}
Let $\mathcal A\in\Q(\PS(X))$ and put $K=\bigcup\mathcal A$. The set $K$ is saturated. Let $\{U_i:i\in I\}$ be an open cover of $K$. For each $L\in\mathcal A$, choose a finite $F_L\subseteq I$ with
\[
   L\subseteq V_L:=\bigcup_{i\in F_L}U_i.
\]
The family $\{\Box_XV_L:L\in\mathcal A\}$ covers $\mathcal A$. Compactness of $\mathcal A$ yields $L_1,\ldots,L_m\in\mathcal A$ such that
\[
   \mathcal A\subseteq\Box_XV_{L_1}\cup\cdots\cup\Box_XV_{L_m}.
\]
It follows that $K\subseteq V_{L_1}\cup\cdots\cup V_{L_m}$, so $K$ is compact.

For $U\in\OO(X)$,
\[
   \mu_X^{-1}(\Box_XU)
   =\left\{\mathcal A:\bigcup\mathcal A\subseteq U\right\}
   =\Box_{\PS(X)}(\Box_XU),
\]
which proves continuity.
\end{proof}

We use the standard convention that a stably compact space is a compact, locally compact, coherent and sober space; see, for example, \cite[Definition~2.2 and Remark~2.4]{Lawson}.

\begin{lemma}[{\cite[Theorem~3.1]{Goubault2010}; see also \cite[Lemma~5.1 and Remark~5.3]{Lawson}}]\label{lem:stable-preserved}
The Smyth power space of a stably compact space is stably compact.
\end{lemma}

A space $X$ is \emph{$k_1$-Hausdorff} (or $k_1$--$T_2$) if every compact subspace is Hausdorff, and $X$ is a \emph{KC-space} if each compact subset of $X$ is closed. The following implications are known; see, for example, \cite{ClontzWilliams}.

\begin{lemma}\label{lem:separation}
\[
   T_2\Longrightarrow k_1\text{--}T_2\Longrightarrow KC\Longrightarrow T_1.
\]
\end{lemma}

\section{Finite Smyth iterates: carriers and compact localization}

For a $T_0$-space $X$, define
\[
   \PS^0(X)=X,
   ~
   \PS^{n+1}(X)=\PS(\PS^n(X))~(n\geq0).
\]
Thus $\PS^1(X)=\PS(X)$. If $Y$ is saturated in $X$, repeated use of Lemma~\ref{lem:saturated-subspace} identifies $\PS^n(Y)$ with a subspace of $\PS^n(X)$.

\begin{lemma}\label{lem:iterate-principal}
Let $K\in\Q(X)$ and $n\geq1$. Then $\PS^{n-1}(K)$ is a compact saturated subset of $\PS^{n-1}(X)$ and
\begin{equation}\label{eq:iterate-principal}
   \PS^n(K)=\up{\PS^n(X)}\PS^{n-1}(K).
\end{equation}
In particular, $\PS^n(K)$ is a saturated subspace of $\PS^n(X)$.
\end{lemma}

\begin{proof}
For $n=1$, the assertion is \eqref{eq:PSK-principal}, and $\PS^0(K)=K\in\Q(X)$. Suppose the assertion holds at level $n$. Recall that the principal saturation $\uparrow_Z z$ of any point $z$ in a topological space $Z$ is compact: if an open family covers $\uparrow_Z z$, one member contains $z$, and, being an upper set, already contains all of $\uparrow_Z z$. Hence $\PS^n(K)$, being the principal saturation of the point $\PS^{n-1}(K)$ in $\PS^n(X)$, is compact saturated. Applying Lemma~\ref{lem:saturated-subspace} inside $\PS^n(X)$ gives
\[
\begin{split}
   \PS^{n+1}(K)
   &=\{\mathcal A\in\Q(\PS^n(X)):\mathcal A\subseteq\PS^n(K)\}\\
   &=\up{\PS^{n+1}(X)}\PS^n(K),
\end{split}
\]
where the last equality uses reverse inclusion at the next Smyth level.
\end{proof}

Put $c_1=\operatorname{id}_{\PS(X)}$. For $n\geq1$, define recursively
\begin{equation}\label{eq:carrier-map}
   c_{n+1}
   =c_n\circ\mu_{\PS^{n-1}(X)}:
   \PS^{n+1}(X)\longrightarrow\PS(X).
\end{equation}
Thus $c_n$ repeatedly takes unions until a compact saturated subset of $X$ remains. In particular,
\begin{equation}\label{eq:carrier-union}
   c_{n+1}(\mathcal A)
   =\bigcup_{A\in\mathcal A}c_n(A).
\end{equation}

\begin{theorem}[Carrier criterion]\label{thm:carrier}
Let $n\geq1$, $A\in\PS^n(X)$ and $K\in\Q(X)$. Then
\begin{equation}\label{eq:carrier-criterion}
   A\in\PS^n(K)
   ~\Longleftrightarrow~
   c_n(A)\subseteq K.
\end{equation}
Consequently, $c_n(A)$ is the unique least member of $\Q(X)$, under inclusion, that carries $A$, in the sense that $A\in\PS^n(c_n(A))$ and
\[
   A\in\PS^n(K)\Longrightarrow c_n(A)\subseteq K.
\]
\end{theorem}

\begin{proof}
The case $n=1$ follows from Lemma~\ref{lem:saturated-subspace}, since $c_1(A)=A$. Suppose the assertion holds at level $n$, and let $\mathcal A\in\PS^{n+1}(X)$. By Lemma~\ref{lem:iterate-principal},
\[
   \mathcal A\in\PS^{n+1}(K)
   ~\Longleftrightarrow~
   \mathcal A\subseteq\PS^n(K).
\]
By the induction hypothesis, this is equivalent to $c_n(A)\subseteq K$ for every $A\in\mathcal A$, which by \eqref{eq:carrier-union} is equivalent to $c_{n+1}(\mathcal A)\subseteq K$.

Taking $K=c_n(A)$ shows that $A\in\PS^n(c_n(A))$. If $A\in\PS^n(K)$ for another $K\in\Q(X)$, then \eqref{eq:carrier-criterion} gives $c_n(A)\subseteq K$.
\end{proof}

\begin{theorem}[Finite compact localization]\label{thm:localization}
Let $X$ be a $T_0$-space and $n\geq1$. The following are equivalent:
\begin{enumerate}[\rm (1)]
\item $\PS^n(X)$ is co-sober.
\item $\PS^n(K)$ is co-sober for every $K\in\Q(X)$.
\end{enumerate}
\end{theorem}

\begin{proof}
Assume first that $\PS^n(X)$ is co-sober. By Lemma~\ref{lem:iterate-principal}, $\PS^n(K)$ is a saturated subspace of $\PS^n(X)$ for every $K\in\Q(X)$. Hence $\PS^n(K)$ is co-sober by Lemma~\ref{lem:cosober-saturated}.

Conversely, assume (2), and let $G\in\KIrr(\PS^n(X))$. As a compact saturated subset of $\PS^n(X)$, the set $G$ is a point of $\PS^{n+1}(X)$. Put
\[
   K=c_{n+1}(G).
\]
The carrier criterion gives $G\in\PS^{n+1}(K)$; equivalently, $G$ is a compact saturated subset of $\PS^n(K)$. By Lemma~\ref{lem:iterate-principal}, $\PS^n(K)$ is saturated in $\PS^n(X)$, and Lemma~\ref{lem:kirr-subspace} therefore shows that $G$ remains $k$-irreducible in $\PS^n(K)$. Co-sobriety of $\PS^n(K)$ gives an $A\in\PS^n(K)$ such that
\[
   G=\up{\PS^n(K)}A.
\]
Since $\PS^n(K)$ is saturated in $\PS^n(X)$ and contains $A$, one has $\up{\PS^n(X)}A\subseteq\PS^n(K)$. Hence the principal saturation of $A$ is the same in the subspace and in the ambient space, and therefore
\[
   G=\up{\PS^n(X)}A.
\]
Thus $\PS^n(X)$ is co-sober.
\end{proof}

\section{Compact Hausdorff carriers}

We first prove the compact Hausdorff case directly. Let $C$ be compact Hausdorff, and let $\mathcal K(C)$ be the set of all nonempty compact subsets of $C$. Since $C$ is $T_1$, every subset is saturated, so $\mathcal K(C)=\Q(C)$ as sets. The full Vietoris topology on $\mathcal K(C)$ is generated by
\[
   \Box U=\{K\in\mathcal K(C):K\subseteq U\},
   ~
   \Diamond U=\{K\in\mathcal K(C):K\cap U\neq\varnothing\},
\]
where $U$ is open in $C$.

\begin{lemma}[{\cite[Theorem~4.9.12]{Michael}}]\label{lem:vietoris-compact}
If $C$ is compact Hausdorff, then its full Vietoris hyperspace $\mathcal K(C)$ is compact Hausdorff.
\end{lemma}

\begin{lemma}\label{lem:family-vietoris-closed}
Let $C$ be compact Hausdorff and $\mathcal A\in\Q(\PS(C))$. Then $\mathcal A$ is closed in the full Vietoris hyperspace $\mathcal K(C)$.
\end{lemma}

\begin{proof}
Saturation in the reverse-inclusion order means that $\mathcal A$ is hereditary under nonempty compact subsets. Let $F\in\mathcal K(C)\setminus\mathcal A$. For each $K\in\mathcal A$, one has $F\nsubseteq K$, for otherwise heredity would imply $F\in\mathcal A$. Choose $x_K\in F\setminus K$.

For every $y\in K$, choose disjoint open neighborhoods $O_y$ of $y$ and $N_y$ of $x_K$. Compactness of $K$ gives $y_1,\ldots,y_s\in K$ with
\[
   K\subseteq O_{y_1}\cup\cdots\cup O_{y_s}.
\]
Thus
\[
   V_K=O_{y_1}\cup\cdots\cup O_{y_s},
   ~
   W_K=N_{y_1}\cap\cdots\cap N_{y_s}
\]
are disjoint open sets satisfying $K\subseteq V_K$ and $x_K\in W_K$. The sets $\Box V_K$ cover $\mathcal A$ in the upper Vietoris topology. Compactness of $\mathcal A$ yields $K_1,\ldots,K_m\in\mathcal A$ such that
\[
   \mathcal A\subseteq\Box V_{K_1}\cup\cdots\cup\Box V_{K_m}.
\]
Then
\[
   N=\bigcap_{j=1}^m\Diamond W_{K_j}
\]
is a full Vietoris neighborhood of $F$. If $L\in\mathcal A$, then $L\subseteq V_{K_j}$ for some $j$, and therefore $L\cap W_{K_j}=\varnothing$. Thus $N\cap\mathcal A=\varnothing$, proving that $\mathcal A$ is closed.
\end{proof}

\begin{theorem}\label{thm:compact-Hausdorff}
Let $C$ be compact Hausdorff and $\mathcal A\in\KIrr(\PS(C))$. Then
\begin{equation}\label{eq:union-generator}
   \mathcal A
   =\up{\PS(C)}\!\left(\bigcup\mathcal A\right).
\end{equation}
Consequently, $\PS(C)$ is co-sober.
\end{theorem}

\begin{proof}
Put $D=\bigcup\mathcal A$. By Lemma~\ref{lem:union-map}, $D$ is a nonempty compact subset of $C$. Hence $D$ is compact Hausdorff. Every member of $\mathcal A$ is contained in $D$, so $\mathcal A\subseteq\PS(D)$. By \eqref{eq:PSK-principal}, $\PS(D)$ is a saturated subspace of $\PS(C)$; Lemmas~\ref{lem:saturated-subspace} and \ref{lem:kirr-subspace} show that $\mathcal A$ is a $k$-irreducible compact saturated subset of $\PS(D)$. By Lemma~\ref{lem:family-vietoris-closed}, $\mathcal A$ is closed in $\mathcal K(D)$.

Suppose that $D\notin\mathcal A$. Since $\mathcal A$ is full-Vietoris closed, there is a basic neighborhood
\[
   \Box V\cap\bigcap_{i=1}^r\Diamond U_i
\]
of the point $D$ that is disjoint from $\mathcal A$, where $V,U_1,\ldots,U_r$ are open in $D$. The condition $D\in\Box V$ forces $V=D$, so the box condition is the whole hyperspace. Moreover, $r\geq1$, because $\mathcal A$ is nonempty. Thus each $U_i$ is nonempty and
\begin{equation}\label{eq:diamond-separation}
   D\in\bigcap_{i=1}^r\Diamond U_i,
   ~
   \mathcal A\cap\bigcap_{i=1}^r\Diamond U_i=\varnothing.
\end{equation}
For $1\leq i\leq r$, let
\[
   \mathcal A_i=\{K\in\mathcal A:K\cap U_i=\varnothing\}.
\]
Then $\mathcal A=\mathcal A_1\cup\cdots\cup\mathcal A_r$. After removing empty terms, each $\mathcal A_i$ is hereditary, and
\[
   \mathcal A_i
   =\mathcal A\cap\bigl(\mathcal K(D)\setminus\Diamond U_i\bigr)
\]
is full-Vietoris closed. Since $\mathcal K(D)$ is compact by Lemma~\ref{lem:vietoris-compact}, $\mathcal A_i$ is compact in the full Vietoris topology and hence in the coarser upper Vietoris topology. Heredity means precisely that it is saturated in $\PS(D)$, so $\mathcal A_i\in\Q(\PS(D))$.

A $k$-irreducible compact saturated set that is a finite union of compact saturated sets equals one of them, by induction on the number of terms. Hence $\mathcal A=\mathcal A_j$ for some $j$. Every member of $\mathcal A$ then misses $U_j$, so their union $D$ misses $U_j$, contradicting \eqref{eq:diamond-separation}. Therefore $D\in\mathcal A$.

Every member of $\mathcal A$ is contained in $D$, while heredity and $D\in\mathcal A$ imply that every nonempty compact subset of $D$ belongs to $\mathcal A$. Hence
\[
   \mathcal A
   =\{K\in\mathcal K(C):K\subseteq D\}
   =\up{\PS(C)}D.
\]
\end{proof}

\begin{corollary}\label{cor:first-general}
Let $X$ be a $T_0$-space such that every compact saturated subspace of $X$ is Hausdorff. Then, for every $\mathcal A\in\KIrr(\PS(X))$,
\[
   \mathcal A
   =\up{\PS(X)}\!\left(\bigcup\mathcal A\right).
\]
In particular, $\PS(X)$ is co-sober.
\end{corollary}

\begin{proof}
Let $D=\bigcup\mathcal A$. By Lemma~\ref{lem:union-map}, $D\in\Q(X)$, and $D$ is Hausdorff by hypothesis. The family $\mathcal A$ lies in the saturated subspace $\PS(D)$ and remains $k$-irreducible there by Lemma~\ref{lem:kirr-subspace}. Apply Theorem~\ref{thm:compact-Hausdorff}. Since $\PS(D)$ is saturated in $\PS(X)$, the principal upward closure is the same in the subspace and in the ambient space.
\end{proof}

\begin{corollary}\label{cor:first-k1}
If $X$ is $k_1$-Hausdorff, then, for every $\mathcal A\in\KIrr(\PS(X))$,
\[
   \mathcal A
   =\up{\PS(X)}\!\left(\bigcup\mathcal A\right).
\]
In particular, $\PS(X)$ is co-sober.
\end{corollary}

\begin{remark}\label{rem:first-direct}
The first-iterate conclusion for Hausdorff and metrizable spaces already follows from Corollary~\ref{cor:first-k1}. In particular, it is independent of the compact-localization theorem for higher iterates and of the stable-compactness result used in the next section.
\end{remark}

\section{Stable compact carriers and finite iterates}

The localization theorem isolates the precise compact input needed for all finite levels.

\begin{theorem}[Finite iterates over stably compact carriers]\label{thm:stable-carriers}
Let $X$ be a $T_0$-space such that every $K\in\Q(X)$ is stably compact as a subspace of $X$. Then $\PS^n(X)$ is co-sober for every finite $n\geq1$.
\end{theorem}

\begin{proof}
Fix $K\in\Q(X)$. By hypothesis, $K$ is stably compact. Lemma~\ref{lem:stable-preserved} implies inductively that $\PS^n(K)$ is stably compact for every finite $n\geq1$. Hence each $\PS^n(K)$ is locally compact, coherent and $T_0$, and is therefore co-sober by Lemma~\ref{lem:lc-coherent}. The result now follows from Theorem~\ref{thm:localization}.
\end{proof}

\begin{corollary}\label{cor:Hausdorff-carriers}
Let $X$ be a $T_0$-space such that every compact saturated subspace of $X$ is Hausdorff. Then $\PS^n(X)$ is co-sober for every finite $n\geq1$.
\end{corollary}

\begin{proof}
Every compact Hausdorff space is stably compact. Apply Theorem~\ref{thm:stable-carriers}.
\end{proof}

\begin{corollary}\label{cor:k1}
If $X$ is $k_1$-Hausdorff, then $\PS^n(X)$ is co-sober for every finite $n\geq1$.
\end{corollary}

\begin{proof}
Every $K\in\Q(X)$ is a compact subspace of $X$ and hence Hausdorff. Apply Corollary~\ref{cor:Hausdorff-carriers}.
\end{proof}

\begin{corollary}\label{cor:Hausdorff-metric}
If $X$ is Hausdorff, then $\PS^n(X)$ is co-sober for every finite $n\geq1$. In particular, $\PS(X)$ is co-sober. The same conclusion holds for every metrizable space.
\end{corollary}

The case $n=1$ gives affirmative answers to the Hausdorff and metrizable questions mentioned in the Introduction.

\begin{corollary}\label{cor:sober-cosober}
If $X$ is Hausdorff, then $\PS^n(X)$ is both sober and co-sober for every finite $n\geq1$.
\end{corollary}

\begin{proof}
Every Hausdorff space is sober. By Theorem~\ref{thm:sob-smyth}, sobriety is preserved by the Smyth power space, so $\PS^n(X)$ is sober for every finite $n\geq1$. Co-sobriety follows from Corollary~\ref{cor:Hausdorff-metric}.
\end{proof}

\begin{example}\label{ex:cocountable}
Let $X$ be an uncountable set with the co-countable topology. Any two nonempty open sets intersect, so $X$ is not Hausdorff. It is nevertheless $k_1$-Hausdorff \cite{ClontzWilliams}. For completeness, every compact subset is finite. Indeed, if $A\subseteq X$ is infinite, choose distinct points $a_1,a_2,\ldots$ in $A$ and put
\[
   U_n=X\setminus\{a_m:m\geq n\}.
\]
The sets $U_n$ form an open cover of $A$ with no finite subcover. Thus every compact subspace is finite and hence Hausdorff.

The space $X$ is not sober: $X$ itself is an irreducible closed set, but, since the topology is $T_1$, it is not the closure of a point. Hence every $\PS^n(X)$ with $n\geq1$ is co-sober by Corollary~\ref{cor:k1} and non-sober by repeated application of Theorem~\ref{thm:sob-smyth}.
\end{example}

\section{Conclusions}

The general preservation problem for co-sobriety under the Smyth power construction has a negative answer even for co-sober $T_1$-spaces \cite{XuXuJiQiu}. The results above identify a broad positive mechanism. Compact localization shows that every finite-level compact obstruction is supported on a canonical least compact saturated subspace of the original space. Thus co-sobriety of finite Smyth iterates is controlled entirely by compact saturated carriers.

The strongest general criterion obtained here is that all compact saturated carriers be stably compact. A convenient sufficient condition is that they be Hausdorff; this strictly extends the Hausdorff case, since every $k_1$-Hausdorff space satisfies it, and Example~\ref{ex:cocountable} is non-Hausdorff. The direct first-iterate argument gives, in addition, an explicit generator formula by taking unions.

The proof also indicates a boundary of the present method. For a general KC-space, compact saturated carriers need not be Hausdorff. The direct Vietoris proof uses Hausdorff separation inside the compact carrier, while the finite-iterate argument uses stable compactness. Neither argument therefore extends automatically to arbitrary KC-spaces. The carrier construction itself, however, is independent of separation assumptions and may be useful in other preservation problems for iterated power spaces.

\section*{Statements and Declarations}

\textbf{Competing interests.} The authors have no relevant financial or non-financial interests to disclose.

\vspace{0.1cm}

\textbf{Data availability.} No datasets were generated or analyzed during the current study.

\vspace{0.1cm}

\textbf{Declaration of generative AI and AI-assisted technologies in the manuscript preparation process.} During the preparation of this work, the authors used ChatGPT (OpenAI) for assisting with language editing and organization of the exposition. All mathematical statements, proofs, references, and the final text were checked and approved by the authors, who take full responsibility for the content.

\end{document}